\documentclass[11pt]{article}

\usepackage{amsmath}
\usepackage{amssymb}
\usepackage{amsthm}
\usepackage{mathdots}
\usepackage{verbatim}

\newcommand{\mf}{\ensuremath{\mathfrak}}
\newcommand{\mb}{\ensuremath{\mathbb}}
\newcommand{\Tr}{\ensuremath{\text{Tr}}}

\numberwithin{equation}{section}
\newtheorem{thm}[equation]{Theorem}

\newtheorem{lem}[equation]{Lemma}
\newtheorem{prop}[equation]{Proposition}

\theoremstyle{definition}

\theoremstyle{remark}

\newtheorem{ex}[equation]{Example}

\title{Compact symmetric spaces, triangular factorization, and Cayley coordinates}
\author{Derek Habermas}

\begin{document}

\maketitle

\begin{abstract}

Let $U/K$ represent a connected, compact symmetric space, where $\theta$ is an involution of $U$ that fixes $K$, $\phi: U/K \to U$ is the geodesic Cartan embedding, and $G$ is the complexification of $U$.  We investigate the intersection of $\phi(U/K)$ with the Bruhat decomposition of $G$ corresponding to a $\theta$-stable triangular, or LDU,  factorization of the Lie algebra of $G$.  When $g \in \phi(U/K)$ is generic, the corresponding factorization $g=ld(g)u$ is unique, where $l \in N^-, d(g) \in H$, and $u \in N^+$.  In this paper we present an explicit formula for $d$ in Cayley coordinates, compute it in several types of symmetric spaces, and use it to identify representatives of the connected components of the generic part of $\phi(U/K)$. This formula calculates a moment map for a torus action on the highest dimensional symplectic leaves of the Evens-Lu Poisson structure on $U/K$.

\end{abstract}

\section{Introduction}

Let $U/K$ be a connected, irreducible, compact, Riemannian symmetric space on which $U$ acts isometrically.  Then $K$ is the fixed point set of an involution $\theta$ of $U$.  Let $G$ be the complexification of $U$ and $\mf g$ the complexification of the Lie algebra $\mf u$ of $U$.  We assume $\theta$ can be extended to a holomorphic involution of $G$ and we let $\theta$ denote this extension as well as the corresponding involutions of $\mf u$ and $\mf g$.  In this paper we consider the intersection of the image of the Cartan embedding 
\begin{equation*}
\phi: U/K \to U \subseteq G : uK \mapsto u\theta(u^{-1})
\end{equation*}
with the Bruhat (or triangular, or LDU) decomposition
\begin{equation*}
G = \coprod_{w \in W} \Sigma_w^G, \qquad \Sigma_w^G = N^-wHN^+
\end{equation*}
relative to a $\theta$-stable triangular decomposition $\mf g = \mf n^- \oplus \mf h \oplus \mf n^+$.

For a generic element $g$ in this intersection, $g \in \Sigma_1^G \cap \phi(U/K)$, this yields a unique triangular factorization $g=ld(g)u$.  Our main contribution is to produce explicit formulas for the diagonal map $d$ in classical cases when $\theta$ is an inner automorphism, using Cayley coordinates. This formula calculates a moment map for a torus action on the highest dimensional symplectic leaves of the Evens-Lu Poisson structure on $U/K$~\cite{Evens-Lu} studied also by Foth and Otto~\cite{Foth-Otto}, and by Caine~\cite{Caine}. This intersection is also studied in the context of harmonic analysis on symmetric spaces; for more information, see papers by Pickrell~\cite{Pickrell2}, and by Borodin and Olshanski~\cite{Borodin-Olshanski}.

For each type of symmetric space under consideration, we choose a representation of $\mf u$ in $\mf s\mf u(n)$ and a specific involution $\theta$ of $\mf g$ such that $\theta$ fixes each of the subspaces $\mf n^-, \mf h,$ and $\mf n^+$ which, in each representation, always consist of strictly lower triangular, diagonal, and strictly upper triangular matrices, respectively.  This is made precise in section 3.

The formulas for $d$ contain determinants such as $\det(1+X)$, where $X$ is in $i\mf p$, the $-1$-eigenspace of $\theta$ acting on the Lie algebra $\mf u$. Due to the relatively sparse nature of these matrices, these determinants are often easily calculable, and we illustrate this with examples.  The structure of the paper is as follows.

In section 2 we introduce notation and review relevant background for the intersection $\Sigma_1^G \cap \phi(U/K)$.

In section 3 we calculate $d$ in Cayley coordinates.

In section 4 we use this calculation to identify representatives in each connected component of $\Sigma_1^G \cap \phi(U/K)$.

In section 5 we show explicit calculations for $d$ in low dimensional examples, and we apply the results of section 4 to five types of compact symmetric spaces.

In the appendix, some non-standard representations used in the paper are more fully explained.


\section{Background} \label{background}

Here we review the intersection of a compact symmetric space with a compatible Bruhat decomposition; this material is presented in more detail in~\cite{Pickrell2}.  As stated in the introduction, $U/K$ is a connected irreducible compact symmetric space, where $U$ is a connected Lie group acting on the symmetric space isometrically and transitively, $G$ is the complexification of $U$, and $K \subseteq U$ is the connected component, containing the identity, of the fixed point set of an involution $\theta$ of $U$.  In a slight abuse of notation, we also use $\theta$ to denote the induced involution on the Lie algebra $\mf u$ of $U$ as well as its complex linear extension to the Lie algebra $\mf g$ of $G$.  We also assume that $\theta$ extends to a holomorphic involution on $G$ which will also be denoted $\theta$.  Let $g \mapsto g^{-*}$ denote the Cartan involution of $G$ fixing $U$, and let $g^\theta$ denote $\theta(g)$; since the inversion map, $*$, and $\theta$ commute, this notation should not cause confusion. Let $G_0$ denote the fixed point set of the involution of $G$ given by $\sigma:g \mapsto g^{-*\theta}$.  

We have Cartan embeddings of symmetric spaces
\[
\begin{matrix}
U/K & \stackrel{\phi}{\to} & U & : uK \mapsto uu^{-\theta} \\
\downarrow & & \downarrow \\
G/G_0 & \stackrel{\phi}{\to} & G & : gG_0 \mapsto gg^{*\theta}.
\end{matrix}
\]
These are totally geodesic embeddings of symmetric spaces.  The following proposition (Theorem 1a in~\cite{Pickrell2}) characterizes the images of these embeddings as subsets of $G$.

\begin{prop} \label{variety}
Let $\phi$ be the Cartan embedding as stated above. Then we have the following inclusion maps,
\[
\begin{matrix}
\phi(U/K) = \{g \in G : g^{-1} = g^* = g^\theta\}_0 & \to & U = \{g \in G : g^{-1} = g^*\}\\
\downarrow & & \downarrow \\
\phi(G/G_0) = \{g \in G : g^* = g^\theta\}_0 & \to & G
\end{matrix}
\]
where $\{ \cdot \}_0$ denotes the connected component containing the identity.

\end{prop}

Let $\mf u = \mf k \oplus i\mf p$ be the decomposition of $\mf u$ into $+1$ and $-1$ eigenspaces of $\theta$.  By Proposition \ref{variety} we can use the derivative of the Cartan embedding to identify the tangent space of $U/K$ at $1K$ with
\[
i\mf p = \{X \in \mf g : -X = X^* = X^\theta \}.
\]
The exponential map of $\mf g$ maps $i\mf p$ onto $\phi(U/K)$.  (See chapter VII of~\cite{Helgason}.)

Fix a maximal abelian subalgebra $\mf t_0 \subseteq \mf k$.  We then obtain $\theta$-stable Cartan subalgebras
\[
\mf h_0 = Z_{\mf g_0}(\mf t_0) = \mf t_0 \oplus \mf a_0, \quad \mf t = \mf t_0 \oplus i\mf a_0, \quad \text{and} \quad \mf h = \mf h_0^{\mathbb C}
\]
of $\mf g_0$, $\mf u$, and $\mf g$, respectively, where $Z_{\mf g_0}(\mf t_0)$ is the centralizer of $\mf t_0$ in $\mf g_0$ and $\mf a_0 \subseteq \mf p$ (see (6.60) of~\cite{Knapp}).  Let $T_0 = exp(\mf t_0)$  and $T = exp(\mf t)$ correspond to maximal tori in $K$ and $U$, respectively.

We obtain a $\theta$-stable triangular decomposition $\mf g = \mf n^- \oplus \mf h \oplus \mf n^+$ so that $\sigma(\mf n^{\pm}) = \mf n^{\mp}$.  (See p. 709 in~\cite{Pickrell2}.)  Let $N^\pm = exp(\mf n^\pm)$ and $H = exp(\mf h)$.  We also let $W = W(G,T)$ denote the Weyl group, $W=N_U(T)/T \cong N_G(H)/H$.  Corresponding to this triangular decomposition of $\mf g$, we have the Bruhat decomposition of the group $G$,
\[
G = \coprod_{w \in W} \Sigma_w^G, \quad \Sigma_w^G = N^-wHN^+,
\]
where $\Sigma_w^G$ is diffeomorphic to $(N^- \cap wN^-w^{-1}) \times H \times N^+$.  Elements in $\Sigma_1^G$ are called ``generic."  Define 
\begin{equation}
d:\Sigma_1^G \to H: g \mapsto d(g) \quad \text{ if } g=ld(g)u \label{defined}
\end{equation}
where $l \in N^-, d(g) \in H$ and $u \in N^+$. Since this factorization is unique for generic elements, the map $d$ is well defined.

Intersecting the Bruhat decomposition of $G$ with $\phi(U/K)$ we obtain a decomposition, indexed by $W$, of the symmetric space.  Let $\Sigma_w^{\phi(U/K)}$ denote $\Sigma_w^G \cap \phi(U/K)$. Theorems 2 and 3 from~\cite{Pickrell2} examine the intersections of the symmetric spaces and varieties mentioned in Proposition \ref{variety} with $\Sigma_w^G$ for arbitrary $w \in W$.  The following proposition (which follows immediately from part (e) of Theorem 2 and Theorem 3 in~\cite{Pickrell2}) summarizes some facts from these theorems about $\Sigma_1^{G}$ and its intersection with $\phi(U/K)$. For any group $\Gamma$, let $\Gamma^{(2)}$ denote $\{ g \in \Gamma \, | \, g^2=1 \}$.

\begin{prop} \label{connected} (a) Each connected component of $\Sigma_1^{\phi(U/K)}$ contains an element ${\bf w} \in T_0^{(2)}$, which is unique up to multiplication by elements in $exp(i\mf a_0)^{(2)}$. 

(b) If ${\bf w} \in T_0^{(2)}$, then ${\bf w} \in \phi(U/K)$ if and only if there exists ${\bf w}_1 \in N_U(T_0)$ such that $\phi({\bf w}_1K) = {\bf w}$. 

(c) If $T \subseteq K$, then $i\mf a_0 = 0$, and so $\pi_0(\Sigma_1^{\phi(U/K)})$ is in one-to-one correspondence with $$\{ {\bf w} \in T_0^{(2)} \, | \, \exists {\bf w}_1  \in N_U(T_0) \text{ such that } \phi({\bf w}_1K) = {\bf w} \}.$$
\end{prop}


\section{The diagonal map in Cayley coordinates} \label{general}

In this section we compute, in Cayley coordinates, the diagonal map $d:\Sigma_1^G \to H$ and its restriction to $\Sigma_1^{\phi(U/K)}$ for compact symmetric spaces $U/K$ of types $AIII$, $DIII$, $CI$, $CII$, and $BDI$. For these, the complexification $G$ of $U$ is either $SL(n,\mb C)$, $SO(n,\mb C)$, or $Sp(n/2,\mb C)$.  Let $N^+_{SL}$, $N^-_{SL}$, and $H_{SL}$ be the subgroups of $SL(n,\mb C)$ consisting of upper triangular unipotent, lower triangular unipotent, and diagonal matrices, respectively.  Let $\mathcal T$ denote the anti-holomorphic involution of $SL(n,\mb C)$ given by $\mathcal T(g) = (g^{-1})^*$ where $*$ denotes conjugate transpose.  Let $\tau$ denote anti-transpose (reflection across the anti-diagonal), the holomorphic anti-involution of $GL(n,\mb C)$ given by $g^\tau = J_n g^t J_n^{-1}$ where $J_n \in GL(n,\mb C)$ has entries equal to 1 on the anti-diagonal and 0 elsewhere.  Also let $\tau$ denote its restriction to any subgroup of $GL(n,\mb C)$ as well as the derivatives acting on the corresponding Lie algebras.

We embed $SO(n,\mb C)$ (respectively, $Sp(n/2,\mb C)$) into $SL(n,\mb C)$ as the fixed point set of a holomorphic involution $\Theta$ of $SL(n,\mb C)$ such that $\Theta$ preserves $N^+_{SL}$, $N^-_{SL}$, and $H_{SL}$, and such that $\mathcal T\Theta = \Theta\mathcal T$.  For $G=SO(n,\mb C)$, define $\Theta(g) = (g^{-1})^\tau$.  For $G=Sp(n/2,\mb C)$, define $\Theta(g) = I_{\frac{n}{2},\frac{n}{2}}(g^{-1})^\tau I_{\frac{n}{2},\frac{n}{2}}^{-1}$, where $I_{\frac{n}{2},\frac{n}{2}}$ is the $n \times n$ identity matrix with the first $n/2$ diagonal entries negated.  Each $\Theta$ has the specified properties; for more information see Appendix A.  For each $G$, let
$$
N^+ = G \cap N^+_{SL}, \quad N^- = G \cap N^-_{SL} \quad \text{and} \quad H = G \cap H_{SL},
$$
and let $\mf n^+$, $\mf n^-$ and $\mf h$ be their Lie algebras.  Then $\mf g = \mf n^+ \oplus \mf h \oplus \mf n^-$ is a triangular decomposition of $\mf g$.

For each type of symmetric space $U/K$ we choose a holomorphic involution $\theta$ of $SL(n,\mb C)$ such that it commutes with both $\mathcal T$ and $\Theta$, the triangular decomposition is $\theta$-stable, and $G_0 = G^{\mathcal T\theta}$.  Note that $\mathcal T\theta = \sigma$ as defined in section \ref{background}.  The restriction of $\theta$ to $G$, the restriction to $U$, and the corresponding involutions of the Lie algebras $\mf g$ and $\mf u$ will still be denoted $\theta$.  The choices for $\theta$ are given the following table.
\begin{equation}
\begin{array}{|c|c|c|} \hline
\text{Type} & U/K & \theta \\ \hline
AIII & SU(m+n)/S(U(m)\times U(n)) & Ad(I_{m,n}) \\ \hline
DIII & SO(2n)/U(n) & Ad(I_{n,n}) \\ \hline
CI & Sp(n)/U(n) & Ad(I_{n,n}) \\ \hline
CII & Sp(p+q)/Sp(p)\times Sp(q) & Ad(I_{p,2q,p}) \\ \hline
BDI & SO(p+q)/SO(p)\times SO(q), \quad p \text{ even} & Ad(I_{\frac{p}{2},q,\frac{p}{2}}) \\ \hline
\end{array} \label{symtable}
\end{equation}
The matrix $I_{a,b}$ (or $I_{a,b,a}$) is a diagonal matrix with the first $a$ diagonal entries $-1$ and the next $b$ diagonal entries $1$ (and the next $a$ diagonal entries $-1$, respectively).  (For type $BDI$, if $p$ and $q$ are both odd, $\theta$ is an outer automorphism of $SO(p+q)$.  We address this case in section \ref{calculations}.) In the discussion that follows, if the symmetric space is not specified, we assume that $\theta = Ad(\hat{I})$ and that matrices have dimension $n \times n$.

Define the Cayley map by
\[
\Phi:\mf u(n) \to \{g \in U(n)| -1 \not\in spec(g)\}:X
\mapsto g=(1-X)(1+X)^{-1}. 
\]
Note that $\Phi$ is invertible by $g \mapsto (1-g)(1+g)^{-1}$.

\begin{lem} \label{addmult} Suppose that $\psi:GL(n,\mb C) \to GL(n,\mb C)$ is an automorphism or an anti-automorphism and that $\psi$ can be extended to a linear operator $\bar{\psi}$ on $End_{\mb C}(\mb C^n)$.  Let $X \in End_{\mb C}(\mb C^n)$ be in the tangent space to $GL(n,\mb C)$ at $1$.  If $-X = X^* = \bar{\psi}(X)$, then $X$ is in the domain of the Cayley map $\Phi$, and
\[
\Phi(X) \in \{ g \in GL(n,\mb C) \, | \, g^{-1} = g^* = \psi(g) \}.
\]
\end{lem}

\begin{proof} Let $\psi$ be an automorphism (or an anti-automorphism) of $GL(n, \mb C)$. Then since $End_{\mb C}(\mb C^n)$ is complete, and $End_{\mb C}(\mb C^n)\backslash GL(n,\mb C)$ has measure zero, we have $\bar{\psi}(XY) = \bar{\psi}(X)\bar{\psi}(Y)$ (or $\bar{\psi}(XY) = \bar{\psi}(Y)\bar{\psi}(X)$) for all $X,Y \in End_{\mb C}(\mb C^n)$.  Now suppose $-X = X^* = \bar{\psi}(X)$ and let $g = \Phi(X)$.  Then $X$ is skew Hermitian and $g$ is unitary.  So,
\begin{align}
\psi(g) & = \bar{\psi}\left( (1-X)(1+X)^{-1} \right) \nonumber \\
& = \bar{\psi} (1-X) (\bar{\psi}(1+X))^{-1} \label{addmult1} \\
& = (1+X)(1-X)^{-1} = g^{-1}. \nonumber
\end{align}

Note that if $\psi$ is an anti-automorphism, \eqref{addmult1} follows from the fact that $1-X$ and $(1+X)^{-1}$ commute.
\end{proof}

\begin{prop} \label{conjugation} Let $U/K$ be one of the symmetric spaces in table \eqref{symtable} with corresponding involution $\theta$.  Then
 $\Phi(i\mf p) \subseteq \phi(U/K)$. 
\end{prop}

\begin{proof} By Proposition \ref{variety}, we must show that, for each $U/K$,
\begin{equation*}
\Phi(i\mf p) \subseteq \{g \in U | g^{-1} = g^{\theta} \}_0.
\end{equation*}
Each involution $\theta$ meets the criteria of Lemma \ref{addmult}.  Therefore, since $i\mf p$ is connected, by continuity of $\Phi$
we have
\begin{equation*}
\Phi(i\mf p) \subseteq \{g \in U(n) | g^{-1} = g^{\theta}\}_0.
\end{equation*}
Furthermore, since the
determinant is fixed under conjugation, we have $\det(g) = \det(g^\theta) = \det(g^{-1}) = (\det(g))^{-1}$ 
which implies that $\det(g) = \pm 1$.  By continuity of $\Phi$, and
since $0 \in i\mf p$, we have $\det(g) = 1$.  So, 
\begin{equation*}
\Phi(i\mf p) \subseteq \{g \in SU(n) | g^{-1}=g^{\theta}\}_0.
\end{equation*}
All that remains to be shown is that $\Phi(i\mf p) \subseteq U$.  In the
case where $U=SU(n)$, we are done.  For $U=SO(n)$, note that $\tau$ meets the criteria of Lemma \ref{addmult}, since our representation of $\mf s\mf o(n)$ lies in the $-1$
 eigenspace of $\tau$.  Therefore,
\begin{equation*}
\Phi(i\mf p) \subseteq \{g \in SU(n) | g^{-1}=g^\tau\} = U.
\end{equation*}
The case where $U=Sp(n/2)$ follows similarly, since our representation of $\mf s\mf p(n/2)$
lies in the $-1$ eigenspace of $Ad(I_{\frac{n}{2},\frac{n}{2}}) \circ \tau$.
\end{proof}

Note that multiplication of a matrix $A$ by $\hat{I} = I_{a,b}$ (or $I_{a,b,a}$) on the left has the effect of negating the first $a$ rows of $A$ and fixing the next $b$ rows (and negating the following $a$ rows, respectively), and multiplication by $\hat{I}$ on the right has the corresponding effect on columns of $A$. Thus, conjugation of $A$ by $\hat{I}$ fixes the top left $a \times a$ block of $A$, negates the $a \times b$ block to its right, and so on. We will refer to these alternately as the blocks fixed by $\theta$ and the blocks negated by $\theta$. For example, in type $AIII$, the $m \times n$ and $n \times m$ off-diagonal blocks of $g \in U = SU(m+n)$ are negated by $\theta$, and the $m \times m$ and $n \times n$ diagonal blocks are fixed by $\theta$, corresponding to $K \cong S(U(m) \times U(n))$. 

To simplify the notation, let $I_{k}$ denote $I_{k, n-k}$ when it is understood from context to be an $n \times n$ matrix; in particular, let $I_0 = 1$ and $I_n = -1$. Also, let $A[k]$ denote the principal $k \times k$ block of the matrix $A$.

\begin{lem} \label{minors} Let $X \in \mf s\mf u(n)$ and $g=\Phi(X)$.  Then, for $1 \le k \le n$,
\begin{equation*}
\det(g[k]) = \frac{\det(1+I_kX)}{\det(1+X)}.
\end{equation*}
\end{lem}

\begin{proof} Write $X = \left[
\begin{array}{c|c}
X_1 & X_2 \\ \hline
X_3 & X_4
\end{array}
\right]$ and $(1+X)^{-1} = \left[
\begin{array}{c|c}
Y_1 & Y_2 \\ \hline
Y_3 & Y_4
\end{array}
\right]$, where $X_1, Y_1 \in M_{k \times k}(\mb C)$, and so on.  Then
\begin{align*}
(1+I_kX)(1+X)^{-1} & = \left[
\begin{array}{c|c}
(1-X_1)Y_1 -X_2Y_2 & * \\ \hline
X_3Y_1 + (1+X_4)Y_3 & X_3Y_2 + (1+X_4)Y_4
\end{array}
\right] \\
& = \left[
\begin{array}{c|c}
g[k] & * \\ \hline
0_{(n-k) \times k} & 1_{(n-k) \times (n-k)}
\end{array}
\right].
\end{align*}
Taking determinants, the claim follows.
\end{proof}

Consider the diagonal map
$$d:\Sigma_1^{SL(n,\mb C)} \to H_{SL}: g \mapsto d(g) \quad \text{ if } g=ld(g)u$$
where $l \in N^-_{SL}, d(g) \in H_{SL}$ and $u \in N^+_{SL}$. Since $G \subseteq SL(n, \mb C)$ for each symmetric space under consideration, we shall let $d$ also denote the restriction of this map to $\Sigma_1^G$ to correspond with \eqref{defined}.

\begin{lem} \label{mainthm}  Let $g \in \Sigma_1^{SL(n, \mb C)}$ such that $g=\Phi(X)$ for some $X \in \mf s\mf u(n)$. Then $\det(1+I_kX) \neq 0$ for $1 \le k \le n$, and
\begin{align*}
d(g) = diag \bigg( & \dfrac{\det(1+I_{1}X)}{\det(1+X)}, \frac{\det(1+I_{2}X)}{\det(1+I_{1}X)}, \frac{\det(1+I_{3}X)}{\det(1+I_{2}X)}, \dots \\
 &  \dots , \frac{\det(1+I_{n}X)}{\det(1+I_{n-1}X)} = \frac{\det(1-X)}{\det(1+I_{n-1}X)} \bigg).
\end{align*}
\end{lem}

\begin{proof}  Let $X\in \mf s\mf u(n)$ such that $g = \Phi(X)$ is generic, and let $g = ld(g)u$ as described above.  Then, by Gaussian elimination, $\det(g[k]) \neq 0$ for $1 \le k \le n$, and 
\begin{equation}
d(g) = diag \left( \det(g[1]), \dfrac{\det(g[2])}{\det(g[1])}, \dfrac{\det(g[3])}{\det(g[2])}, \dots , \dfrac{\det(g)}{\det(g[n-1])} \right). \label{gauss}
\end{equation}
Then the proposition follows immediately from Lemma \ref{minors}.
\end{proof}


\section{Identifying $T_0^{(2)} \cap \phi(U/K)$ with Cayley coordinates} \label{boldwsection}

In this section, motivated by Proposition \ref{connected}, we use Cayley coordinates to explicitly identify $T_0^{(2)} \cap \phi(U/K)$ for each type of symmetric space in table \ref{symtable} using Lemma \ref{mainthm}. This is not completely straightforward, as no ${\bf w} \in T_0^{(2)} \cap \phi(U/K)$ but the identity is in the image of the Cayley map $\Phi$.

\begin{ex} \label{S2e} Let $U/K = SU(2)/U(1) \cong \mb CP^1 \cong S^2$.  Let
\begin{equation*}
X =
\begin{bmatrix}
0 & z \\
-\bar{z} & 0
\end{bmatrix}
\in i\mf p \subseteq \mf s\mf u(n).
\end{equation*}
and let $g=\Phi(X)$.  If $g$ is generic, then
\begin{equation*}
d(g)=
\begin{bmatrix}
\frac{1-|z|^2}{1+|z|^2} & 0 \\
0 & \frac{1+|z|^2}{1-|z|^2}
\end{bmatrix}.
\end{equation*}
There are two connected components of $\Sigma_1^{\phi(SU(2)/U(1))}$.  By Proposition \ref{connected}, these are indexed by
\begin{equation*}
T_0^{(2)} = \{ \pm \left[
\begin{array}{rr}
 1 & 0 \\
 0 & 1
\end{array}
\right]\}
\end{equation*}
Obviously, $\Phi(0)=+1$, but $-1 \not\in \Phi(i\mf p)$.  However, letting $|z|$ tend to infinity, we see that
\begin{equation*}
\lim_{|z|\to \infty} \Phi(
\begin{bmatrix}
0 & z \\
-\bar{z} & 0
\end{bmatrix}
) = \lim_{|z|\to \infty} 
\begin{bmatrix}
\frac{1-|z|^2}{1+|z|^2} & 0 \\
0 & \frac{1+|z|^2}{1-|z|^2}
\end{bmatrix}
= 
\begin{bmatrix}
-1 & 0 \\
0 & -1
\end{bmatrix}.
\end{equation*}
\end{ex}
Since $\phi(SU(2)/U(1))$ is connected and complete, this calculation verifies that $-1 \in \phi(SU(2)/U(1))$. Note that $\det(X)$ appears (up to a sign) as a summand in $\det(1+I_kX)$. The next theorem generalizes this technique.

First we need some notation to let us talk precisely about submatrices. Let $A$ be an $n \times n$ matrix, and let $0 \le l \le n$.  Let $Q_{l,n}$ denote the set of all subsets of $\{1, \dots , n\}$ with cardinality $l$.  Let $\alpha = \{i_1, \dots ,i_l\} \in Q_{l,n}$, and let $A[\alpha]$ denote the $l \times l$ submatrix consisting of the intersection of rows $i_1, \dots ,i_l$ and columns $i_1, \dots ,i_{l}$ of $A$.

Now, viewing $A$ as an operator on $\mb C^n$, from Fredholm theory we have
\begin{align*}
\det (1 + A) & = \sum_{l=0}^n \Tr (\wedge^l(A)) \\
& = \sum_{l=0}^n \sum_{1 \le i_1 < \dots < i_l \le n} \langle \wedge^l(A)e_{i_1} \wedge \dots \wedge e_{i_l},\, e_{i_1} \wedge \dots \wedge e_{i_l} \rangle \\ 
& = \sum_{l=0}^n \sum_{\alpha \in Q_{l,n}} \det A[\alpha]. 
\end{align*}
(For convenience, we define $\det A[\emptyset] = 1$.) Applying this calculation to the result of Lemma \ref{mainthm}, we find the $k^{th}$ entry of $d(g)$ can be written
\begin{equation}
[d(g)]_{kk} = \frac{\det (1+I_kX)}{\det (1+I_{k-1}X)} = \frac{\displaystyle{\sum_{l=0}^{n} \sum_{\alpha \in Q_{l,n}} \det (I_kX)[\alpha]}}{\displaystyle{\sum_{l=0}^{n} \sum_{\alpha \in Q_{l,n}} \det (I_{k-1}X)[\alpha]}}. \label{dformula}
\end{equation}
This shows that the non-zero entries of $d(g)$ are ratios of sums of determinants of submatrices of $I_kX$, for $0 \le k \le n$. Furthermore, since multiplication of $X$ by $I_k$ on the left negates the first $k$ rows of $X$, $\det(1 + I_kX)$ has the same summands for each $0 \le k \le n$, up to a sign.

Next we recall some facts from linear algebra and Lie theory, and establish notation. Recall that, since $T_0 \subseteq H_{SL}$, it consists of diagonal matrices, and so $T_0^{(2)} = \{ {\bf w} \in T_0 \, | \, {\bf w}^2 = 1 \}$ consists of diagonal matrices whose diagonal entries are $\pm 1$. Let ${\bf w} \in T_0^{(2)}$. If ${\bf w}$ has exactly $l$ entries equal to $-1$, then define $\alpha_{\bf w} \in Q_{l,n}$ by $i \in \alpha_{\bf w}$ if and only if $[{\bf w}]_{ii} = -1$. Then multiplication of a diagonal matrix $A$ by ${\bf w}$ (on the left or the right) negates the $i^{th}$ diagonal entry of $A$ if and only if $i \in \alpha_{\bf w}$. Also, if $\theta = Ad({\hat I})$, define $\alpha_\theta$ by $i \in \alpha_\theta$ if and only if $[\hat{I}]_{ii} = -1$. It follows that $[A]_{i,j}$ is in a block negated by $\theta$ if and only if either $i \in \alpha_\theta$ and $j \not\in \alpha_\theta$, or $i \not\in \alpha_\theta$ and $j \in \alpha_\theta$. Finally, the Weyl group of $U$ acts on $T$ by conjugation; that is, for $w_1 \in W$, let ${\bf w}_1 \in w_1$, then conjugation of a (diagonal) element $A \in T_0$ by ${\bf w}_1$ performs a permutation $\sigma_1$ on the diagonal entries of $A$. In particular, ${\bf w}_1$ can be obtained by performing $\sigma_1$ on the rows of an element of $T_0$, possibly with sign changes.

\begin{thm} \label{boldwthm} Let $U/K$ be one of the symmetric spaces in table \ref{symtable} with $\theta = Ad(\hat{I})$ where $\hat{I}$ has the form $I_{a,b}$ or $I_{a,b,a}$, and let ${\bf w} \in T_0^{(2)}$ with the block structure induced by $\theta$. Then ${\bf w} \in \phi(U/K)$ if and only if the number of $-1$ entries in the $b \times b$ diagonal block of ${\bf w}$ is the same as the number of $-1$ entries outside the $b \times b$ diagonal block. More precisely, ${\bf w} \in \phi(U/K)$ if and only if exactly half of $\alpha_{\bf w}$ is contained in $\alpha_\theta$; that is, $|\alpha_{\bf w} \cap \alpha_\theta | = |\alpha_{\bf w} \backslash \alpha_\theta |$.
\end{thm}

\begin{proof}
Suppose ${\bf w} \in \phi(U/K)$. By proposition \ref{connected} there exists ${\bf w}_1 \in N_U(T_0)$ such that ${\bf w} = {\bf w}_1{\bf w}_1^{-\theta} = {\bf w}_1\hat{I}{\bf w}_1^{-1}\hat{I}^{-1}$, and so ${\bf w}\hat{I} = {\bf w}_1\hat{I}{\bf w}_1^{-1}$. Since conjugation of $\hat{I}$ by ${\bf w}_1$ permutes the diagonal entries of $\hat{I}$, thus fixing the number of negative entries of $\hat{I}$, and since multiplication of $\hat{I}$ by ${\bf w}$ changes the sign of $[\hat{I}]_{i,i}$ if and only if $i \in \alpha_{\bf w}$, we have $|\alpha_{\bf w} \cap \alpha_\theta | = |\alpha_{\bf w} \backslash \alpha_\theta |$.

Conversely, let ${\bf w} \in T_0^{(2)}$ such that $|\alpha_{\bf w} \cap \alpha_\theta | = |\alpha_{\bf w} \backslash \alpha_\theta |$. We shall construct $X \in i\mf p$ such that $\lim_{t\to \infty} d(\Phi(tX)) = {\bf w}$. This will suffice, as $\Phi(i\mf p) \subseteq \phi(U/K)$, and $\phi(U/K)$ is complete.

Let $s=|\alpha_{\bf w} \cap \alpha_\theta | = |\alpha_{\bf w} \backslash \alpha_\theta |$. If $s=0$, then ${\bf w} = 1 \in \phi(U/K)$, so assume $s \ge 1$. Let $\alpha_{\bf w} \cap \alpha_\theta = \{i_1, \dots ,i_s\}$ and $\alpha_{\bf w} \backslash \alpha_\theta = \{j_1, \dots ,j_s\}$ such that they are each enumerated in ascending order. Whether $\hat{I}$ has the form $I_{a,b}$ or $I_{a,b,a}$, $[{\bf w}]_{i_r,i_r} = -1$ is in an $a \times a$ diagonal block, and $[{\bf w}]_{j_r,j_r} = -1$ is in the $b \times b$ diagonal block, for all $1 \le r \le s$.


Case 1: $\mf u = \mf s\mf u(n) = \{X \in \mf s\mf l(n,\mb C) | -X = X^*\}$. Then $U/K$ is of type $AIII$ and $\theta = Ad(I_{m,n})$. Choose $X \in \mf s\mf l(n,\mb C)$ by $[X]_{i_r,j_r} = 1$ and $[X]_{j_r,i_r} = -1$ for all $1 \le r \le s$, with all other entries zero. Then $X$ is skew-Hermitian by construction, and all non-zero entries of $X$ are in blocks negated by $\theta$, since $i_r \in \alpha_\theta$ and $j_r \not\in \alpha_\theta$ for all $1 \le r \le s$. Therefore, $-X = X^* = X^\theta$; that is, $X \in i\mf p$.

Let $t \in \mb R$. Then $[d(\Phi(tX))]_{k,k} = \det (1+tI_kX) / \det (1 + tI_{k-1}X)$ for $1 \le k \le n$ by Lemma \ref{mainthm}. By \eqref{dformula}, this is a ratio of polynomials in $t$ whose terms are identical up to a sign. By construction, $X[\alpha_{\bf w}]$ is the largest submatrix of $X$ with non-zero determinant. Thus, the leading term of $\det(1+tI_kX)$ is $\det(tI_kX)[\alpha_{\bf w}] = \pm t^{2s}$. Furthermore, the leading terms of the numerator and denominator of $[d(\Phi(tX))]_{k,k}$, $\det(tI_kX)[\alpha_{\bf w}]$ and $\det(tI_{k-1}X)[\alpha_{\bf w}]$, respectively, differ by a sign exactly when $k \in \alpha_{\bf w}$. Hence, 
\begin{align*}
\lim_{t\to \infty} [d(\Phi(tX))]_{k,k} &= \frac{\lim_{t \to \infty} \det(tI_kX)[\alpha_{\bf w}]}{\lim_{t \to \infty} \det(tI_{k-1}X)[\alpha_{\bf w}]} \\
&= \begin{cases} -1 \quad \text{ if } k \in \alpha_{\bf w} \\ 1 \quad \text{ if } k \not\in \alpha_{\bf w} \end{cases} \\
&= [{\bf w}]_{k,k}
\end{align*}
As $d$ and $\Phi$ are continuous, ${\bf w} = \lim_{t\to \infty} d(\Phi(tX)) \in \phi(U/K)$.

Case 2: $\mf u = \mf s\mf p(n/2) = \{X \in \mf s\mf l(n,\mb C) | -X = X^* = Ad(I_{\frac{n}{2},\frac{n}{2}})X^{\tau}\}$ (see Appendix A). Suppose $U/K$ is of type $CI$; then $\theta = Ad(I_{\frac{n}{2},\frac{n}{2}})$. Choose $X \in \mf s\mf l(n,\mb C)$ by $[X]_{i_r,j_r} = 1$ and $[X]_{j_r,i_r} = -1$ with all other entries zero. As above, we have $-X = X^* = X^\theta$, so we must show that $-X = Ad(I_{\frac{n}{2},\frac{n}{2}})X^{\tau}$. Note that since ${\bf w} \in Sp(n/2)$, ${\bf w}^* = I_{\frac{n}{2},\frac{n}{2}}^{\hspace{0pt}}{\bf w}^{\tau}I_{\frac{n}{2},\frac{n}{2}}^{-1} = {\bf w}^{\tau\theta}$. Also, as diagonal blocks are fixed by $\theta$ and ${\bf w}$ is diagonal and real, ${\bf w} = {\bf w}^\tau$; that is, ${\bf w}$ is symmetric across the anti-diagonal. Since $\theta = Ad(I_{\frac{n}{2},\frac{n}{2}})$, it follows immediately that $\alpha_{{\bf w}^\tau} \cap \alpha_\theta = \alpha_{\bf w} \setminus \alpha_\theta$ and $\alpha_{{\bf w}^\tau} \setminus \alpha_\theta = \alpha_{\bf w} \cap \alpha_\theta$. Since $\alpha_{\bf w} \cap \alpha_\theta$ and $\alpha_{\bf w} \setminus \alpha_\theta$ are enumerated in ascending order, $X=X^\tau$, and so $-X = X^{\tau\theta} = Ad(I_{\frac{n}{2},\frac{n}{2}})X^\tau$. Therefore, $X \in i\mf p$.

Suppose $U/K$ is of type $CII$, then $\theta = Ad(I_{p,2q,p})$. (So $n=2p+2q$.) Choose $X \in \mf s\mf l(n,\mb C)$ by $[X]_{i_r,j_r} = 1$ and $[X]_{j_r,i_r} = -1$ with all other entries zero. Again, we have $-X = X^* = X^\theta$, and we must show that $-X = Ad(I_{\frac{n}{2},\frac{n}{2}})X^{\tau}$. First note that, as in type $CI$, ${\bf w}={\bf w}^\tau$, and so the $2q \times 2q$ middle block of ${\bf w}$ is symmetric across the anti-diagonal. Thus, $s$ is even, and it follows that $i_1, \dots ,i_{\frac{s}{2}}, j_1, \dots ,j_{\frac{s}{2}} \le n/2$, and $i_{\frac{s}{2}+1}, \dots ,i_s, j_{\frac{s}{2}+1}, \dots ,j_s > n/2$. Hence, the non-zero entries of $X$ are in blocks fixed by $Ad(I_{\frac{n}{2},\frac{n}{2}})$. Furthermore, since $\theta = Ad(I_{p,2q,p})$, in contrast to the type $CI$ case we have $\alpha_{{\bf w}^\tau} \cap \alpha_\theta = \alpha_{\bf w} \cap \alpha_\theta$ and $\alpha_{{\bf w}^\tau} \setminus \alpha_\theta = \alpha_{\bf w} \setminus \alpha_\theta$. Thus, by construction, $-X=X^\tau=Ad(I_{\frac{n}{2},\frac{n}{2}})X^\tau$. Therefore, $X \in i\mf p$.

By the same argument as in case 1, ${\bf w} = \lim_{t\to \infty} d(\Phi(tX)) \in \phi(U/K)$.

Case 3: $\mf u = \mf s\mf o(n) = \{X \in \mf s\mf l(n,\mb C) | -X = X^* = X^\tau\}$ (see Appendix A). Suppose $U/K$ is of type $BDI$, where $n=p+q$ and $p$ is even. Then $\theta = Ad(I_{\frac{p}{2},q,\frac{p}{2}})$. Choose $X \in \mf s\mf l(n, \mb C)$ by $[X]_{i_r,j_r} = 1$ and $[X]_{j_r,i_r} = -1$ for all $1 \le r \le s$ with all other entries zero. Then $-X = X^* = X^\theta$. The argument here is similar to the one for type $CII$; $s$ is even, and since $\alpha_{{\bf w}^\tau} \cap \alpha_\theta = \alpha_{\bf w} \cap \alpha_\theta$ and $\alpha_{{\bf w}^\tau} \setminus \alpha_\theta = \alpha_{\bf w} \setminus \alpha_\theta$, we have $-X = X^\tau$. So $X \in i\mf p$.

Suppose $U/K$ is of type $DIII$; then $\theta = Ad(I_{\frac{n}{2},\frac{n}{2}})$. We have ${\bf w} ={\bf w}_1{\bf w}_1^\theta$ where ${\bf w}_1 \in w_1 \in W$. It is well known that the Weyl group of $U=SO(2n)$ acts by even permutations, in this case, on the diagonal entries of elements in $T_0$, via conjugation. Since $w_1$ has order two, it can be written as a product of disjoint transpositions. As each transposition in $w_1$ corresponds to two negative entries of ${\bf w}$, it follows that $s$ is even. Therefore, choose $X \in \mf s\mf l(n, \mb C)$ by 
\[
[X]_{i_r,j_r} = \begin{cases} 1 \text{ if } i_r \le \frac{s}{2} \\ -1 \text{ if } i_r > \frac{s}{2} \end{cases} \qquad [X]_{j_r,i_r} = \begin{cases} -1 \text{ if } i_r \le \frac{s}{2} \\ 1 \text{ if } i_r > \frac{s}{2} \end{cases} 
\]
with all other entries of $X$ zero. By construction, $-X=X^*=X^\theta$. Since $\theta = Ad(I_{\frac{n}{2},\frac{n}{2}})$, we have $\alpha_{{\bf w}^\tau} \cap \alpha_\theta = \alpha_{\bf w} \setminus \alpha_\theta$ and $\alpha_{{\bf w}^\tau} \setminus \alpha_\theta = \alpha_{\bf w} \cap \alpha_\theta$. Thus, $-X = X^\tau$, and so $X \in i\mf p$.
As in the first two cases, ${\bf w} = \lim_{t\to \infty} d(\Phi(tX)) \in \phi(U/K)$.
\end{proof}


\section{Explicit Calculations Of $d$ and $T_0^{(2)} \cap \phi(U/K)$} \label{calculations}

We now apply the results of sections \ref{general} and \ref{boldwsection} for each type of symmetric space in table \eqref{symtable}.  Throughout this section, $X \in i\mf p$ and $g = \Phi(X)$. As noted in section \ref{boldwsection}, in each representation, each ${\bf w} \in T_0^{(2)}$ is a diagonal matrix whose diagonal entries are $\pm 1$.

\subsection{Type AIII}

Symmetric space: $SU(m+n)/S(U(m) \times U(n)) \cong
Gr(m,\mb C^{m+n})$  \\ 
Involution: $\theta:X \mapsto Ad(I_{m,n})(X) \quad$ \\
Block structure:
\begin{equation}
\mf k = \left\{ \left[ 
\begin{array}{c|c} 
 A & 0 \\ \hline
 0 & B 
\end{array} 
\right] \, | \text{ trace} = 0 \right\} \qquad i\mf p = \left\{ \left[
\begin{array}{c|c}
 0 & Z \\ \hline
 -Z^* & 0
\end{array}
\right] \right\} \label{compgrassip}
\end{equation}
where $A \in \mf u(m)$, $B \in \mf u(n)$, and $Z \in M_{m \times n}(\mb C)$.  Note that $Z$ in \eqref{compgrassip} is the graph coordinate for $Gr(n,\mb C^{m+n})$. 

For example, if $m=1$ then $U/K \cong \mb CP^n$, so
\begin{equation*}
X = \left[
\begin{array}{c|c}
 0 & Z \\ \hline
-Z^* & 0
\end{array}
\right] \quad \text{where } Z=\begin{bmatrix} z_1 & \dots & z_n \end{bmatrix}.
\end{equation*}
By \eqref{dformula}, we have $[d(g)]_{1,1} = (1 - \sum |z_{i}|^2)/(1 + \sum |z_{i}|^2)$, and
\begin{equation*}
[d(g)]_{k,k} = 
\frac{1 + \sum\limits_{i=1}^{k-1} |z_{i}|^2 - \sum\limits_{j=k}^{n}
  |z_{j}|^2}{1 + \sum\limits_{i=1}^{k-2} |z_{i}|^2 -
  \sum\limits_{j=k-1}^{n} |z_{j}|^2}, \qquad 2 \le k \le n+1. 
\end{equation*}

For $\mb CP^1 \cong S^2$, we have $i\mf p \cong \mb C$.  In Cayley coordinates, the above formula yields $[d(g)]_{1,1}=(1-|z|^2)/(1+|z|^2)$, which is the height function in stereographic coordinates (under projection from the south pole) or in the $z$ coordinate on the Riemann sphere.

By Theorem \ref{boldwthm}, ${\bf w} \in T_0^{(2)} \cap \phi(U/K)$ if and only if the number of $-1$ entries in the $m \times m$ upper diagonal block is equal to the number in the $n \times n$ lower diagonal block; there are $\begin{pmatrix} m+n \\ m \end{pmatrix}$ such ${\bf w}$.

\subsection{Type CI}

Symmetric Space: $Sp(n)/U(n)$ \\
Involution: $\theta:X \mapsto Ad(I_{n,n})(X) \quad$  \\
Block structure:
\begin{equation*}
\mf k = \left\{ \left[ 
\begin{array}{c|c} 
 A & 0 \\ \hline
 0 & -A^\tau 
\end{array} 
\right] \right\} \qquad i\mf p = \left\{ \left[
\begin{array}{c|c}
 0 & Z \\ \hline
 -Z^* & 0
\end{array}
\right] \right\}
\end{equation*}
where $A \in \mf u(n)$, and $Z \in M_{n \times n}(\mb C)$ such that $Z = Z^\tau$.

This is a subspace of $Gr(n,\mb C^{2n})$; the condition $-X = Ad(I_{n,n})X^\tau$ restricts $\Phi(X)$ to $Sp(n)$. For example, if $n=2$, then
\begin{equation*}
X=\left[
\begin{array}{c|c}
 0 & Z \\ \hline
 -Z^* & 0
\end{array}
\right], \quad \text{where } Z=
\begin{bmatrix}
 z_{11} & z_{12} \\
 z_{21} & z_{11}
\end{bmatrix},
\end{equation*}
and by \eqref{dformula}, we have
\begin{align*}
d(g)=diag( & \tfrac{1-|z_{12}|^2+|z_{21}|^2-\det
  ZZ^*}{1+2|z_{11}|^2+|z_{12}|^2+|z_{21}|^2+\det ZZ^*}, \tfrac{1-2|z_{11}|^2-|z_{12}|^2-|z_{21}|^2+\det
  ZZ^*}{1-|z_{12}|^2+|z_{21}|^2-\det ZZ^*},\\
& \tfrac{1-|z_{12}|^2+|z_{21}|^2-\det
  ZZ^*}{1-2|z_{11}|^2-|z_{12}|^2-|z_{21}|^2+\det ZZ^*},  \tfrac{1+2|z_{11}|^2+|z_{12}|^2+|z_{21}|^2+\det ZZ^*}{1-|z_{12}|^2+|z_{21}|^2-\det ZZ^*}).
\end{align*}

As noted in the proof of Theorem \ref{boldwthm}, if ${\bf w} \in T_0^{(2)}$ then ${\bf w}^\tau={\bf w}$. Therefore, all ${\bf w}$ satisfying the condition of Theorem \ref{boldwthm} that are symmetric across the anti-diagonal are in $T_0^{(2)} \cap \phi(Sp(n)/U(n))$; there are $2^n$ such ${\bf w}$.

\subsection{Type CII}
Symmetric Space: $Sp(p+q)/Sp(p) \times Sp(q) \cong Gr(p,\mb H^{p+q})$ \\
Involution: $\theta:X \mapsto Ad(I_{p,2q,p})X \quad$  \\
Block Structure:
\begin{equation*}
\mf k = \left\{
\begin{bmatrix}
 A & 0 & 0 & B \\
 0 & C & D & 0 \\
 0 & -D^* & -C^\tau & 0 \\
-B^* & 0 & 0 & -A^\tau
\end{bmatrix}
\right\}, \  i\mf p = \left\{
\begin{bmatrix}
 0 & Z_1 & Z_2 & 0 \\
-Z_1^* & 0 & 0 & Z_2^\tau \\
-Z_2^* & 0 & 0 & -Z_1^\tau \\
 0 & -Z_2^{*\tau} & Z_1^{*\tau} & 0
\end{bmatrix}
\right\}
\end{equation*}
where $A=-A^*, B=B^\tau, C=-C^*, D=D^\tau, Z_1,Z_2 \in M_{p \times q}(\mb C)$.  That is, $\left[
\begin{array}{cc}
A & B \\
-B^* & -A^\tau
\end{array}
\right] \in \mf s\mf p(p)$, and $\left[
\begin{array}{cc}
C & D \\
-D^* & -C^\tau
\end{array}
\right] \in \mf s\mf p(q)$.  So $\mf k \cong \mf s\mf p(p) \oplus \mf s\mf p(q)$. 

For example, if $p=q=1$, then $U/K \cong \mb HP^1$, so
\begin{equation*}
X = \left[
\begin{array}{c|c|c}
 0 & Z & 0 \\ \hline
-Z^* & 0 & -Z^\tau \\ \hline
0 & Z^{*\tau} & 0
\end{array}
\right] \quad \text{where } Z= \begin{bmatrix} z_1 & z_2 \end{bmatrix},
\end{equation*}
and, using \eqref{dformula}, $d(g)$ simplifies to\\
\begin{align*}
d(g)=diag\big(& \tfrac{1-|z_1|^2-|z_2|^2}{1+|z_1|^2+|z_2|^2}, \tfrac{1+|z_1|^2-|z_2|^2+(|z_1|^2+|z_2|^2)^2}{1-(|z_1|^2+|z_2|^2)^2},\\
& \tfrac{1-(|z_1|^2+|z_2|^2)^2}{1+|z_1|^2-|z_2|^2 +
  (|z_1|^2+|z_2|^2)^2}, \tfrac{1+|z_1|^2+|z_2|^2}{1-|z_1|^2-|z_2|^2}\big).
\end{align*}

By Theorem \ref{boldwthm} we have ${\bf w} \in T_0^{(2)} \cap \phi(Sp(p+q)/Sp(p) \times Sp(q))$ if and only if ${\bf w}$ is symmetric across the anti-diagonal (as in type $CI$), and has an equal (even) number of $-1$ entries in the $2q \times 2q$ center block, the ``$Sp(q)$ part," as in the $p \times p$ outer blocks combined, the ``$Sp(p)$ part." There are $\begin{pmatrix} p+q \\ p \end{pmatrix}$ such ${\bf w}$.

\subsection{Type DIII}

Symmetric Space: $SO(2n)/U(n)$ \\
Involution: $\theta:X \mapsto Ad(I_{n,n})(X) \quad$  \\
Block structure:
\begin{equation*}
\mf k = \left\{ \left[ 
\begin{array}{c|c} 
 A & 0 \\ \hline
 0 & -A^\tau 
\end{array} 
\right] \right\} \qquad i\mf p = \left\{ \left[
\begin{array}{c|c}
 0 & Z \\ \hline
 -Z^* & 0
\end{array}
\right] \right\}
\end{equation*}
where $A \in \mf u(n)$, and $Z \in M_{n \times n}(\mb C)$ such that $Z = -Z^\tau$.

This is also a subspace of $Gr(n,\mb C^{2n})$; the condition $-X = X^\tau$ restricts $\Phi(X)$ to $SO(2n)$. In particular, the anti-diagonal entries of $X$ must be zero for all $X \in i\mf p$. For example, if $n=3$, we have
\begin{equation*}
X=\left[
\begin{array}{c|c}
 0 & Z \\ \hline
 -Z^* & 0
\end{array}
\right], \quad \text{where } Z=
\begin{bmatrix}
 z_{11} &  z_{12} & 0 \\
 z_{21} & 0 & -z_{12} \\
 0 & -z_{21} & -z_{11}
\end{bmatrix},
\end{equation*}
and, using \eqref{dformula}, $d(g)$ simplifies to \\
\begin{align*}
d(g) = diag\big(& \tfrac{1-|z_{11}|^2+|z_{21}|^2-|z_{12}|^2}{1+|z_{11}|^2+|z_{21}|^2+|z_{12}|^2}, \tfrac{(1+|z_{11}|^2+|z_{21}|^2-|z_{12}|^2)(1-|z_{11}|^2-|z_{21}|^2-|z_{12}|^2)}{(1-|z_{11}|^2+|z_{21}|^2-|z_{12}|^2)(1+|z_{11}|^2+|z_{21}|^2+|z_{12}|^2)}, \\
& \tfrac{1-|z_{11}|^2-|z_{21}|^2-|z_{12}|^2}{1+|z_{11}|^2+|z_{21}|^2-|z_{12}|^2}, \tfrac{1+|z_{11}|^2+|z_{21}|^2-|z_{12}|^2}{1-|z_{11}|^2-|z_{21}|^2-|z_{12}|^2},\\
&
\tfrac{(1-|z_{11}|^2+|z_{21}|^2-|z_{12}|^2)(1+|z_{11}|^2+|z_{21}|^2+|z_{12}|^2)}{(1+|z_{11}|^2+|z_{21}|^2-|z_{12}|^2)(1-|z_{11}|^2-|z_{21}|^2-|z_{12}|^2}, \tfrac{1+|z_{11}|^2+|z_{21}|^2+|z_{12}|^2}{1-|z_{11}|^2+|z_{21}|^2-|z_{12}|^2}\big).
\end{align*}

By Theorem \ref{boldwthm}, for ${\bf w} \in T_0^{(2)} \cap \phi(SO(2n)/U(n))$, there are as many $-1$ entries in the first diagonal block as in the second.  Also, as noted in the proof of Theorem \ref{boldwthm}, $|\alpha_{\bf w} \cap \alpha_\theta |$ $|\alpha_{\bf w} \setminus \alpha_\theta |$ are even. Therefore, ${\bf w} \in T_0^{(2)} \cap \phi(SO(2n)/U(n))$ if and only if ${\bf w}$ is symmetric across the anti-diagonal, and has an even number of $-1$ entries in each $n \times n$ diagonal block; there are $2^{n-1}$ such ${\bf w}$.

\subsection{Type BDI} \label{realgrass}
Symmetric Space: $SO(p+q)/SO(p)\times SO(q) \cong Gr(p,\mb R^{p+q})$ \\ 
Involution: $\theta:X \mapsto  Ad(\hat{I})X \quad$ (Inner if and only if $pq$ is even)

Case 1: $p$ and $q$ are not both odd.  Without loss of generality, assume $p$ is even.  Then $\hat{I} = I_{\frac{p}{2}, q, \frac{p}{2}}$.

\noindent Block structure:
\begin{equation*}
\mf k = \left\{
\begin{bmatrix}
 A & 0 & B \\
 0 & C & 0 \\
-B^* & 0 & -A^\tau
\end{bmatrix}
\right\}, \quad i\mf p = \left\{
\begin{bmatrix}
 0 & Z & 0 \\
-Z^* & 0 & -Z^\tau \\
 0 & Z^{*\tau} & 0
\end{bmatrix}
\right\} 
\end{equation*}
where $A=-A^*, B=-B^\tau, C=-C^*=-C^\tau, Z \in M_{\frac{p}{2} \times q}(\mb C)$.  That is, $C \in \mf s\mf o(q)$, and $\left[
\begin{array}{cc}
A & B \\
-B^* & -A^\tau
\end{array}
\right] \in \mf s\mf o(p)$.  So $\mf k \cong \mf s\mf o(p) \oplus
\mf s\mf o(q)$.

For example, if $p=6$ and $q=1$, then $U/K \cong \mb RP^{6}$, so
\begin{equation*}
X=\left[
\begin{array}{c|c|c}
 0 & Z & 0 \\ \hline
 -Z^* & 0 & -Z^\tau \\ \hline
 0 & Z^{*\tau} & 0
\end{array}
\right], \quad \text{where } Z= \begin{bmatrix} z_3 \\ z_2 \\ z_1 \end{bmatrix}.
\end{equation*}
By \eqref{dformula}, we have
\begin{align*}
d(g)=diag\big(& \tfrac{1 + 2|z_1|^2 + 2|z_2|^2}{1 + 2|z_1|^2 + 2|z_2|^2 +
  2|z_3|^2}, \tfrac{1 + 2|z_1|^2}{1 + 2|z_1|^2 + 2|z_2|^2}, \tfrac{1}{1+2|z_1|^2}, 1, \\
& \tfrac{1+2|z_1|^2}{1}, \tfrac{1+2|z_1|^2+2|z_2|^2}{1+2|z_1|^2}, \tfrac{1+2|z_1|^2+2|z_2|^2+2|z_3|^2}{1+2|z_1|^2+2|z_2|^2}\big).
\end{align*}

The form of $i\mf p$ here is similar to that of the quaternionic Grassmannian, type $CII$; ${\bf w} \in T_0^{(2)} \cap \phi(SO(p+q)/SO(p)\times SO(q))$ if and only if ${\bf w}$ is symmetric across the anti-diagonal and has an equal (even) number of $-1$ entries in the middle $q \times q$ block, the ``$SO(q)$ part," as in the outer $p/2 \times p/2$ outer blocks combined, the ``$SO(p)$ part."    (Notice that if $q$ is odd, the middle diagonal entry must be positive $1$.) There are $\begin{pmatrix} p/2 + \lfloor q/2 \rfloor \\ p/2 \end{pmatrix}$ such ${\bf w}$.

If we restrict our attention to even-dimensional real projective space, $SO(2n+1)/SO(2n)\times SO(1) \cong \mb RP^{2n}$, then by the reasoning above, there can be no $-1$ entries in any ${\bf w} \in T_0^{(2)}$ in $\phi(\mb RP^{2n})$.  Thus, the only ${\bf w}$ present for $\mb RP^{2n}$ is the identity matrix, verifying that $\Sigma_1^{\phi(\mb RP^{2n})}$ is connected.

\bigskip

Case 2: $p$ and $q$ are both odd.  Now $\theta = Ad(\hat{I})$ where $\hat{I} =$
\begin{equation*}
\left[
\begin{array}{cccccccccccccc}
 & & & & & & & & & & & & &  \\
 & 1_{\frac{p-1}{2}\times \frac{p-1}{2}} & & & 0 & & 0 & 0 & & 0 & & & 0 & \\
 & & & & & & & & & & & & &  \\
 & 0 & & & -1_{\frac{q-1}{2}\times \frac{q-1}{2}} & & 0 & 0 & & 0 & & & 0 & \\
 & & & & & & & & & & & & &  \\
 & 0 & & & 0 & & 0 & 1 & & 0 & & & 0 & \\
 & 0 & & & 0 & & 1 & 0 & & 0 & & & 0 & \\
 & & & & & & & & & & & & &  \\
 & 0 & & & 0 & & 0 & 0 & & -1_{\frac{q-1}{2}\times \frac{q-1}{2}} & & & 0 & \\
 & & & & & & & & & & & & &  \\
 & 0 & & & 0 & & 0 & 0 & & 0 & & & 1_{\frac{p-1}{2}\times \frac{p-1}{2}} & \\
 & & & & & & & & & & & & & 
\end{array}
\right]
\end{equation*}
The automorphism $\theta$ is an outer automorphism, since $\lambda\hat{I} \not\in SO(p+q)$ for all $\lambda \in \mb C$.  Still, $\theta$ meets the criteria of Lemma \ref{addmult}, and so the proof of Proposition \ref{conjugation} may be applied to this case.

\noindent Block Structure:
\begin{equation*}
\mf k = \left\{ \left[
\begin{array}{cccccccccccccc}
 & & & & & & & & & & & & &  \\
 & A & & & 0 & & u & u & & 0 & & & B & \\
 & & & & & & & & & & & & &  \\
 & 0 & & & C & & -v & v & & D & & & 0 & \\
 & & & & & & & & & & & & &  \\
 & -u^* & & & v^* & & 0 & 0 & & -v^\tau & & & -u^\tau & \\
 & -u^* & & & -v^* & & 0 & 0 & & v^\tau & & & -u^\tau & \\
 & & & & & & & & & & & & &  \\
 & 0 & & & -D^* & & v^{*\tau} & -v^{*\tau} & & -C^\tau & & & 0 & \\
 & & & & & & & & & & & & &  \\
 & -B^* & & & 0 & & u^{*\tau} & u^{*\tau} & & 0 & & & -A^\tau & \\
 & & & & & & & & & & & & & 
\end{array}
\right] \right\}
\end{equation*}
where $A=-A^*, C=-C^*, B=-B^\tau, D=-D^\tau$.  That is, \\
$
\begin{bmatrix}
A & u & B \\
-u^* & 0 & -u^\tau \\
-B^* & u^{*\tau} & -A^\tau
\end{bmatrix}
 \in \mf s\mf o(p)$, and $
\begin{bmatrix}
C & v & D \\
-v^* & 0 & -v^\tau \\
-D^* & v^{*\tau} & -C^\tau
\end{bmatrix}
 \in \mf s\mf o(q)$.  
\begin{equation*}
i\mf p = \left\{ \left[
\begin{array}{cccccccccccccc}
 & & & & & & & & & & & & &  \\
 & 0 & & & Z_1 & & w_1 & -w_1 & & Z_2 & & & 0 & \\
 & & & & & & & & & & & & &  \\
 & -Z_1^* & & & 0 & & w_2 & w_2 & & 0 & & & -Z_2^\tau & \\
 & & & & & & & & & & & & &  \\
 & -w_1^* & & & -w_2^* & & is & 0 & & -w_2^\tau & & & w_1^\tau & \\
 & w_1^* & & & -w_2^* & & 0 & -is & & -w_2^\tau & & & -w_1^\tau & \\
 & & & & & & & & & & & & &  \\
 & -Z_2^* & & & 0 & & w_2^{*\tau} & w_2^{*\tau} & & 0 & & & -Z_1^\tau & \\
 & & & & & & & & & & & & &  \\
 & 0 & & & Z_2^{*\tau} & & -w_1^{*\tau} & w_1^{*\tau} & & Z_1^{*\tau} & & & 0 & \\
 & & & & & & & & & & & & & 
\end{array}
\right] \right\}
\end{equation*}
where $Z_1,Z_2 \in M_{\frac{p-1}{2} \times \frac{q-1}{2}}(\mb C),\ w_1 \in \mb C^{\frac{p-1}{2}},\ w_2 \in \mb C^{\frac{q-1}{2}}$.

The form we have chosen for $i\mf p$ reveals the presence of $i\mf a_0$ in the center two diagonal entries. For example, if $p=5$ and $q=1$ then $U/K \cong \mb RP^5$, so by \eqref{dformula} we have $d(g)=$
\[
diag\left( \tfrac{1+s^2+4|z_1|^2}{1+s^2+4|z_1|^2+4|z_2|^2}, \tfrac{1+s^2}{1+s^2+4|z_1|^2}, \tfrac{1+is}{1-is}, \tfrac{1-is}{1+is}, \tfrac{1+s^2+4|z_1|^2}{1+s^2}, \tfrac{1+s^2+4|z_1|^2+4|z_2|^2}{1+s^2+4|z_1|^2} \right).
\] 

By Proposition \ref{connected}, the connected components of $\Sigma_1^{\phi(U/K)}$ are indexed by $T_0^{(2)}/exp(i\mf a_0)^{(2)}$.  Therefore, to identify a unique representative for each connected component, we set $s=0$; that is, we require that the middle two diagonal entries of ${\bf w}$ are positive. These representatives of elements in $T_0^{(2)}/exp(i\mf a_0)^{(2)}$ are those ${\bf w}$ that are symmetric across the anti-diagonal and have the same number of $-1$ entries in the inner $\lfloor q/2 \rfloor \times \lfloor q/2 \rfloor$ blocks, the ``$O(q)$ part," as in the outer $\lfloor p/2 \rfloor \times \lfloor p/2 \rfloor$ blocks, the ``$O(p)$ part." There are $\begin{pmatrix} \lfloor p/2 \rfloor + \lfloor q/2 \rfloor \\ \lfloor p/2 \rfloor \end{pmatrix}$ such ${\bf w}$.

It follows that for odd-dimensional real projective space, $\mb RP^{2n+1} \cong SO(2n+2)/(SO(1)\times SO(2n+1))$,  $\Sigma_1^{\phi(\mb RP^{2n+1})}$ is connected.


\appendix

\section{$\theta$-Stable Representations}

We are motivated to use the following representations of $\mf s\mf o(n)$
and $\mf s\mf p(n)$ in $\mf s\mf u(n)$ because they are the fixed
point sets of involutions that preserve the triangular decomposition
of $\mf s\mf l(n,\mb C) = \mf n^- \oplus \mf h \oplus \mf n^+$ where $\mf h$
consists of diagonal matrices, and $\mf n^+$ ($\mf n^-$) consists of
upper (lower) triangular matrices. 

Let $\tau:\mf s\mf l(n,\mb C) \to \mf s\mf l(n,\mb C)$ be the anti-transpose map given by reflection across the anti-diagonal; that is, $X^{\tau} = J X^t J^{-1}$ where $J$ is the $n \times n$ matrix whose entries are ones on the anti-diagonal and zeros elsewhere.  Then $X \mapsto -X^{\tau}$ is an involution of $\mf s\mf l(n,\mb C)$ that stabilizes the above triangular decomposition.  The restriction to $\mf s\mf u(n)$ is also such an involution.

\begin{prop} $\mf s\mf o(n) \cong \{X\in\mf s\mf u(n)|-X^\tau=X\}$, and \\
                        $\mf s\mf p(n) \cong \{X\in\mf s\mf u(2n)|-X^\tau=Ad(I_n)X\}$. 
\end{prop}

\begin{proof} Define involutions on $\mf s\mf l(n,\mb C)$ by
\begin{equation*}
\Theta_0(X) = -X^t, \quad \Theta_1(X) = -X^\tau = -JX^tJ, \quad \text{for} \quad X \in \mf s\mf l(n,\mb C)
\end{equation*}
Then $\Theta_1 = Ad_J \circ \Theta_0$.  Let $P = \frac{1}{\sqrt{2}}(J + i1)$.  A straightforward calculation shows that $\Theta_1 = Ad_P \circ \Theta_0 \circ Ad_{P^{-1}}$.  Thus the Lie algebra isomorphism $Ad_P: \mf s\mf l(n,\mb C) \to \mf s\mf l(n,\mb C)$ maps the fixed point subalgebra of $\Theta_0$ to that of $\Theta_1$.  This completes the proof for $\mf s\mf o(n)$; the proof for $\mf s\mf p(n)$ follows similarly.  
\end{proof}

\noindent {\it Note}: This representation of $\mf s\mf o(n)$ is the space
of infinitesimal isometries of the following $n$ (real) dimensional
subspace of $\mb C^n$:
\begin{equation*}
\{ \left[
\begin{array}{c}
x_n - ix_1 \\
\vdots \\
x_1 - ix_n
\end{array}
\right] |\ x_j \in \mb R \}.
\end{equation*}

\bibliographystyle{alpha.bst}
\bibliography{habermas2}

\end{document}